\documentclass[11pt]{amsart}

\usepackage{amsfonts}
\usepackage{amsthm}
\usepackage{amsmath, amssymb}
\usepackage{thmtools, thm-restate}
\usepackage{enumerate}
\usepackage{graphicx}
\usepackage{color}
\usepackage[percent]{overpic}

\usepackage[colorlinks,linkcolor=blue,citecolor=blue]{hyperref}

\usepackage[top=2.5cm, bottom=2.5cm, left=2.5cm, right=2.5cm]{geometry}
\newtheorem{thm}{Theorem}
\newtheorem{prop}[thm]{Proposition}
\newtheorem{corol}[thm]{Corollary}
\newtheorem{lem}[thm]{Lemma}
\newtheorem*{remark}{Remark}

\newcommand{\Z}{\mathbb{Z}}
\newcommand{\R}{\mathbb{R}}

\newcommand{\Q}{\mathbb{Q}}

\newcommand{\gtop}{g_4^{\rm top}}
\newcommand{\galg}{g_{\rm alg}}

\newcommand{\Int}{{\rm int}\,}

\newcommand{\BlockMatrix}[1]{
\begin{matrix}
#1 & \dotsb & #1\\
\vdots & & \vdots\\
#1 & \dotsb & #1
\end{matrix}
}
\newcommand{\RowMatrix}[1]{
\begin{matrix}
#1 & \dotsb & #1
\end{matrix}
}
\newcommand{\ColumnMatrix}[1]{
\begin{matrix}
#1\\
\vdots\\
#1
\end{matrix}
}

\newcommand{\Legendre}[2]{\left(\frac{#1}{#2}\right)}

\begin{document}

\title{Null-homologous twisting and the algebraic genus}

\author{Duncan McCoy}
\address{Department of Mathematics \\
         The University of Texas At Austin}

\begin{abstract}
The algebraic genus of a knot is an invariant that arises when one considers upper bounds for the topological slice genus coming from Freedman's theorem that Alexander polynomial one knots are topologically slice. This paper develops null-homologous twisting operations as a tool for studying the algebraic genus and, consequently, for bounding the topological slice genus above. As applications we give new upper bounds on the algebraic genera of torus knots and satellite knots.
\end{abstract}

\maketitle
\section{Introduction}
In this paper we study the algebraic genus $\galg(L)$ of an oriented link $L$ in $S^3$, as defined by Feller-Lewark \cite{Feller18Classical}. It is a famous theorem of Freedman that a knot $K$ in $S^3$ with Alexander polynomial $\Delta_K =1$ is topologically slice \cite{Freedman82topology}. It was first observed by Rudolph that this can be used to construct upper bounds on the topological slice genus of knots even when the Alexander polynomial is non-trivial \cite{Rudolph84Projective}. If a knot $K$ has a Seifert surface $F$ containing a subsurface $F'$ such that $\partial F'$ is a knot with Alexander polynomial one, then $F'$ can be replaced by a locally flat disk in the 4-ball to show that $K$ cobounds a locally flat surface of genus $g(F)-g(F')$. The algebraic genus can be defined as the optimal upper bound for $\gtop(L)$ that can be achieved by this method:
\[\galg(L)=\min \left\lbrace g(F)-g(F') \,\middle| \, \parbox{0.6\textwidth}{$F$ is a Seifert surface for $L$ and $F'\subset F$ is a subsurface such that $\partial \Sigma'=K'$ is a knot with $\Delta_{K'}(t)=1$.} \right\rbrace.\]
The main utility of the algebraic genus is that it has several equivalent formulations, including one that depends only on the $S$-equivalence class of the Seifert form of $L$ \cite{Feller18Classical}. This flexibility means that the algebraic genus has been used to prove a variety of results about the topological slice genus \cite{Baader17stable_alternating, Feller16degree, Feller16twobridge, Liechti16positive}. It turns out that, at least for knots, the algebraic genus has a pleasing topological interpretation as the minimal possible genus of a compact, locally flatly embedded surface $F\subseteq B^4$ such that $\partial F =K$ and $\pi_1(B^4\setminus F)\cong \Z$ \cite{Feller19balanced}.

The purpose of this paper is to explore how the algebraic genus changes under certain twisting operations. Using these operations, we obtain new upper bounds for the algebraic genus of satellite knots and torus knots.

\subsection*{Null-homologous twisting}
Given an oriented knot or link $L$ in $S^3$ and an integer $n$, we perform a {\em null-homologous $n$-twist} by taking an unknotted curve $C$ disjoint from $L$ with $lk(C,L)=0$ and performing $1/n$-surgery on $C$. Such a twist can always be performed locally by adding $n$ full twists on some number of parallel strands with appropriate orientations. See Figure~\ref{fig:sample_operation}, for example. 

It turns out that certain pairs of null-homologous twisting operations change the algebraic genus by at most one.
\begin{restatable}{thm}{algtwisting}\label{thm:alg_twisting}
If $L$ and $L'$ are oriented links related by a null-homologous $m$-twist and a null-homologous $n$-twist for $m,n \in \Z$ such that $-mn$ is a square, then
\[
|\galg(L)-\galg(L')|\leq 1.
\]
\end{restatable}
Most notably this shows that for any integer $n$, a single null-homologous $n$-twist changes the algebraic genus by at most one. It is also shows that a null-homologous $+1$-twist and a null-homologous $-1$-twist change the algebraic genus by at most one. This latter observation can be seen as an analog of the well-known fact that changing a negative crossing and a positive crossing changes the smooth slice genus by at most one.

\begin{figure}
  \begin{overpic}[width=0.3\textwidth]{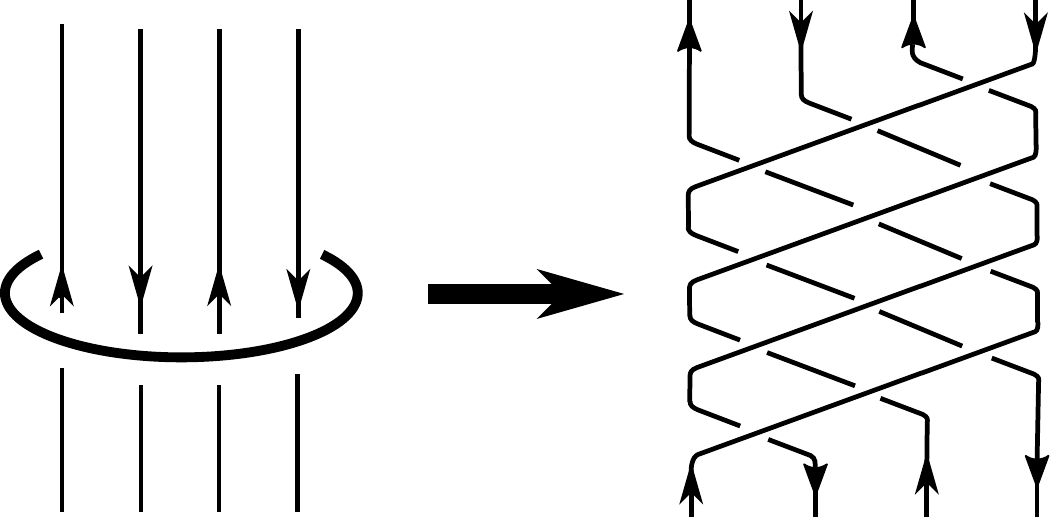}
    \put (-12,17) {$-1$}
  \end{overpic}
  \caption{A negative null-homologous $-1$-twist on $4$ strands.}
  \label{fig:sample_operation}
  \end{figure}

For any link one can always find pairs of null-homologous $+1$- and $-1$-twists which decrease the algebraic genus. This leads to the following description of the algebraic genus.
\begin{restatable}{thm}{twistcharacterization}\label{thm:twist_characterization}  For any link $L$, we have
\[\galg(L)=\min \left\lbrace \max\{n,p\} \,\middle| \, \parbox{0.6\textwidth}{$L$ can be converted to a link $L'$ with $\galg(L')=0$ by $p$ null-homologous $+1$-twists and $n$ null-homologous $-1$-twists.} \right\rbrace.\]
\end{restatable}
For knots, a stronger formulation of Theorem~\ref{thm:twist_characterization} holds. Using the work Borodzik and Friedl on the Blanchfield form \cite{Borodzik14Algebraic, Borodzik15UnknottingI}, one can show that the null-homologous twists in Theorem~\ref{thm:twist_characterization} can be be realized by crossing changes  \cite{Feller19balanced}.

The condition that $-mn$ be a square turns out to be essential to the proof Theorem~\ref{thm:alg_twisting}.
\begin{restatable}{prop}{othertwisting}\label{prop:other_twisting}
For any $m,n\in \Z$ such that $-mn$ is not a square, there is a knot $K$ with $\galg(K)=\gtop(K)=2$, which can be unknotted by performing a null-homologous $m$-twist and a null-homologous $n$-twist. 
\end{restatable}

\subsection*{Satellite knots}
For satellite knots we prove the following upper bound on the algebraic genus. This bound was first obtained (using different ideas) by Feller, Miller and Pinzon-Caicedo \cite{Feller19satellite}.
\begin{restatable}{thm}{satellite}\label{thm:satellites}
For a satellite knot $P(K)$, we have
\[
\galg(P(K))\leq \galg(P(U))+ \galg(K).
\]
\end{restatable}
One striking feature of Theorem~\ref{thm:satellites} is that the upper bound it establishes is independent of the winding number of the pattern $P$. This behaviour should be contrasted with that of both the classical Seifert genus and the smooth slice genus where dependence on the winding number of the pattern is unavoidable. For example, if one takes $K_n$ to be the $(1,n)$-cable of the trefoil, then one can  show that $g(K_n)=g_4(K_n)=n$. However, it follows from Theorem~\ref{thm:satellites} that $\galg(K_n)=\gtop(K_n)=1$.

It is natural to wonder whether there is an analogue of Theorem~\ref{thm:satellites} for the topological slice genus. A detailed discussion of this question and related issues can be found in \cite{Feller19satellite}.

\subsection*{Torus knots}
Whilst the smooth slice genera of torus knots have now been determined by a variety of methods, the topological slice genus remains far less well understood. Rudolph showed that in general the topological slice genus of a torus knot is strictly smaller than the 3-ball genus \cite{Rudolph84Projective}. Later Baader-Feller-Lewark-Liechti constructed further upper bounds on the topological slice genus of torus knots, showing that with the exception of torus knots with $|\sigma(T_{p,q})|=2g_4(T_{p,q})$ the topological slice genus satisfies $\gtop(T_{p,q})\leq \frac67 g_4(T_{p,q})$ \cite{Baader18topological}.

Using null-homologous twisting operations we establish the following upper bound.
\begin{restatable}{thm}{topgenus}\label{thm:top_genus}
For any torus knot or link $T_{p,q}$ with $p,q>1$ we have
\[
\gtop(T_{p,q})\leq \galg(T_{p,q})< \frac{pq}{3}+ p\log_2 q + q\log_2 p.
\]
\end{restatable}
This bound is particularly effective when $p$ and $q$ are both relatively large.
One can measure the asymptotic difference between the smooth and topological slice genera of torus knots by considering the following limit:
\begin{equation*}
\ell := \lim_{\min\{p,q\}\rightarrow \infty} \frac{\gtop(T_{p,q})}{g_4(T_{p,q})}.
\end{equation*}
It is known that this limit exists and satisfies the bounds $\frac{1}{2}\leq \ell< \frac{3}{4}$ \cite{Baader18topological}. Theorem~\ref{thm:top_genus} provides an improved upper bound for $\ell$ by showing that $\ell\leq \frac{2}{3}$.

\subsection*{Structure} In Section~\ref{sec:properties} we set out the properties of the algebraic genus that will be used throughout the paper and prove Theorem~\ref{thm:alg_twisting}. In Section~\ref{sec:decreasing}, we show that there is always a null-homologous $+1$-twist and a $-1$-twist that can be used to decrease the algebraic genus. This gives the proof of Theorem~\ref{thm:twist_characterization}. Then in Section~\ref{sec:satellite} and Section~\ref{sec:torus} contain the results on the algebraic genera of satellite knots and torus knots respectively. Finally we conclude with Section~\ref{sec:anisotropic} where we prove Proposition~\ref{prop:other_twisting}

\subsection*{Acknowledgements}
The statement of Theorem~\ref{thm:satellites} was first found and proven by Peter Feller, Allison N. Miller and Juanita Pinzon-Caicedo \cite{Feller19satellite}. The author is grateful to Peter Feller for bringing this result to his attention and for other illuminating conversations.

\section{Properties of the algebraic genus}\label{sec:properties}
In this section we recap some of the necessary properties of the algebraic genus and prove Theorem~\ref{thm:alg_twisting}. Throughout this paper, all knots and links will be oriented. A Seifert surface for a link $L$ is a connected, oriented, embedded surface $F\subseteq S^3$ with $\partial F=L$. If $L$ has $r$ components, then a genus $g$ Seifert surface has $H_1(F;\Z)\cong \Z^{2g+r-1}$. A Seifert surface comes equipped with its Seifert form $\theta: H_1(F;\Z)\times H_1(F;\Z) \rightarrow \Z$. A subgroup $H\leq H_1(F;\Z)$ of rank $2n$ is said to be Alexander trivial, if for some (equivalently any) basis, the matrix $M$ representing $\theta|_{H}$ has the property that $\det(tM-M^T)=t^n$. We record the following three equivalent definitions of the algebraic genus. The equality of all three quantities were essentially proven by Feller-Lewark, where the third quantity is a variation of their characterization of the algebraic genus in terms of 3-dimensional cobordism distance \cite{Feller18Classical}. 

\begin{prop}\label{prop:3char}
Let $L\subset S^3$ be an oriented $r$-component link. The algebraic genus can be characterized in the following equivalent ways:
\begin{enumerate}
\item \[\galg(L)=\min \left\lbrace g(F)-g(F') \,\middle| \, \parbox{0.5\textwidth}{$F$ is a Seifert surface for $L$ and $F'\subseteq F$ is a subsurface such that $\partial F'=K'$ is a knot with $\Delta_{K'}(t)=1$.} \right\rbrace\]

\item \[\galg(L)=\min \left\lbrace \frac{m-r+1}{2}-n \,\middle| \, \parbox{0.5\textwidth}{$L$ has a Seifert form $\theta: \Z^m \times \Z^m \rightarrow \Z$ with an Alexander trivial subgroup of rank $2n$.} \right\rbrace\]

\item \[\galg(L)=\min \left\lbrace \frac{n-r+1}{2} \,\middle| \, \parbox{0.5\textwidth}{$L$ can be obtained by $n$ oriented bands moves on a knot $K'$ with $\Delta_{K'}(t)=1$.} \right\rbrace\]
\end{enumerate}
\end{prop}

The following lemma shows the equivalence of the first two definitions. We refer the reader to \cite[Proposition~9]{Feller18Classical} for proof.
\begin{lem}\label{lem:subgroups_to_surfaces}
Given a link $L$ with a Seifert surface $F$ of genus $g$ and corresponding Seifert form $\theta$. There is an Alexander trivial subgroup of rank $2n$ in $H_1(F;\Z)$ if and only if $F$ contains a connected genus $n$ subsurface $F'$, where $\partial F'=K'$ is a knot with $\Delta_{K'}=1$.  \qed
\end{lem}
Although it is not known in general which Seifert surfaces for a link realize the algebraic genus, it turns out that any Seifert surface can be stabilized until it realizes the algebraic genus. The following is a consequence of the results of \cite[Section~2]{Feller18Classical}

\begin{lem}\label{lem:surface_stabilization} Let $L$ be an oriented link with $r$ components and let $F$ be a Seifert surface for $L$. Then $F$ can be stabilized to yield a surface $\widetilde{F}$ containing a subsurface $\widetilde{F}'$ such that $\partial \widetilde{F}'$ is a knot with Alexander polynomial one and $\galg(L)=g( \widetilde{F})-g(\widetilde{F}')$. \qed
\end{lem}

The following lemma shows how the algebraic genus changes under oriented band moves.
\begin{lem}\label{lem:band_addition}\label{lem:alg_band_attach}
Let $L$ be an $r$ component link and $L'$ an $r+1$ component link related by an oriented band move. Then
\[
\galg(L')\leq \galg(L)\leq \galg(L') +1.
\]
\end{lem}
\begin{proof}
Suppose that $L_2$ is obtained from $L_1$ by an oriented band move, where $L_1$ has $m$ components and $L_2$ has $m\pm 1$ components. Choose a connected Seifert surface $F$ for $L_1$ disjoint from the band $B$ realizing the band move being performed. This is always possible, since we may choose a diagram for $L_1$ so that the band $B$ appears as a short planar band between two strands. Applying Seifert's algorithm to such a diagram yields a Seifert surface which can be made disjoint from $B$. Furthermore, Lemma~\ref{lem:surface_stabilization} shows that by stabilizing we may assume that $F$ realizes the algebraic genus. Thus $F$ contains a subsurface $F'$ such that $\partial F$ is a knot with Alexander polynomial one and $\galg(L_1)=g(F)-g(F')$. Take $F''$ to be the Seifert surface for $L_2$ obtained by attaching the band $B$ to $F$. Clearly $F'$ is still a subsurface of $F''$ so \[\galg(L_2)\leq g(F'') - g(F').\]
If $L_2$ has $m+1$ components, then $g(F'')=g(F)$. Taking $L=L_1$ and $L'=L_2$ in this case shows that $\galg(L')\leq \galg(L)$. If $L_2$ has $m-1$ components, then $g(F'')=g(F)+1$ . Taking $L=L_2$ and $L'=L_1$ in this case shows that $\galg(L)\leq \galg(L')+1$. This proves the two required inequalities.
\end{proof}

\begin{proof}[Proof of Proposition~\ref{prop:3char}]
The equality of the first two definitions follows from Lemma~\ref{lem:subgroups_to_surfaces}. We prove equality between the first and third definitions. Suppose that a link $L$ can be obtained from an Alexander polynomial one knot $K'$ by $n$ oriented band moves. Suppose that $n_+$ of these band moves increase the number of components and $n_-$ of the moves decrease the number of components. Since $n_+ + n_-=n$ and $n_+-n_- = r-1$, we see that $n_-=\frac{n-r+1}{2}$. Thus Lemma~\ref{lem:band_addition} shows $\galg(L')\leq \frac{n-r+1}{2}$. Conversely, suppose that $F$ is a Seifert surface for $L$ and $F'$ a subsurface cobounding an Alexander polynomial one knot $K'$ realizing the algebraic genus. Consider the surface $\Sigma=F \setminus \Int F'$. The surface $\Sigma$ can be constructed by starting with $K'$ and attaching bands. Since $g(\Sigma)=\galg(L)= g(F)-g(F')$, the surface $\Sigma$ can be constructed by attaching oriented $2\galg(L) + r-1$ bands to $K'$. Thus $L$ can be constructed from $K'$ by $2\galg(L) + r-1$ band moves, as required.
\end{proof}

We conclude the section by proving Theorem~\ref{thm:alg_twisting}.
\algtwisting*
\begin{proof}
Suppose that $L'$ is obtained from $L$ by a null-homologous $m$-twist and a null-homologous $n$-twist. That is, $L'$ is obtained from $L$ by performing $m$-surgery and $n$-surgery on two unknotted curves, say $C_1$ and $C_2$. First we construct a nice Seifert surface for $L$. As shown in Figure~\ref{fig:surface_choice}, we can choose a diagram for $L$ such that Seifert's algorithm yields a Seifert surface $F$ for $L$ which is disjoint from $C_1$ and $C_2$. Moreover we can choose a basis $H_1(F;\Z)$ such that the classes linking non-trivially with $C_1$ and $C_2$ can be represented by a collection of disjoint curves forming an unlink. Lemma~\ref{lem:surface_stabilization} shows that by further stabilizing $F$ we can assume that it realizes the algebraic genus of $L$.
\begin{figure}
  \begin{overpic}[width=0.9\textwidth]{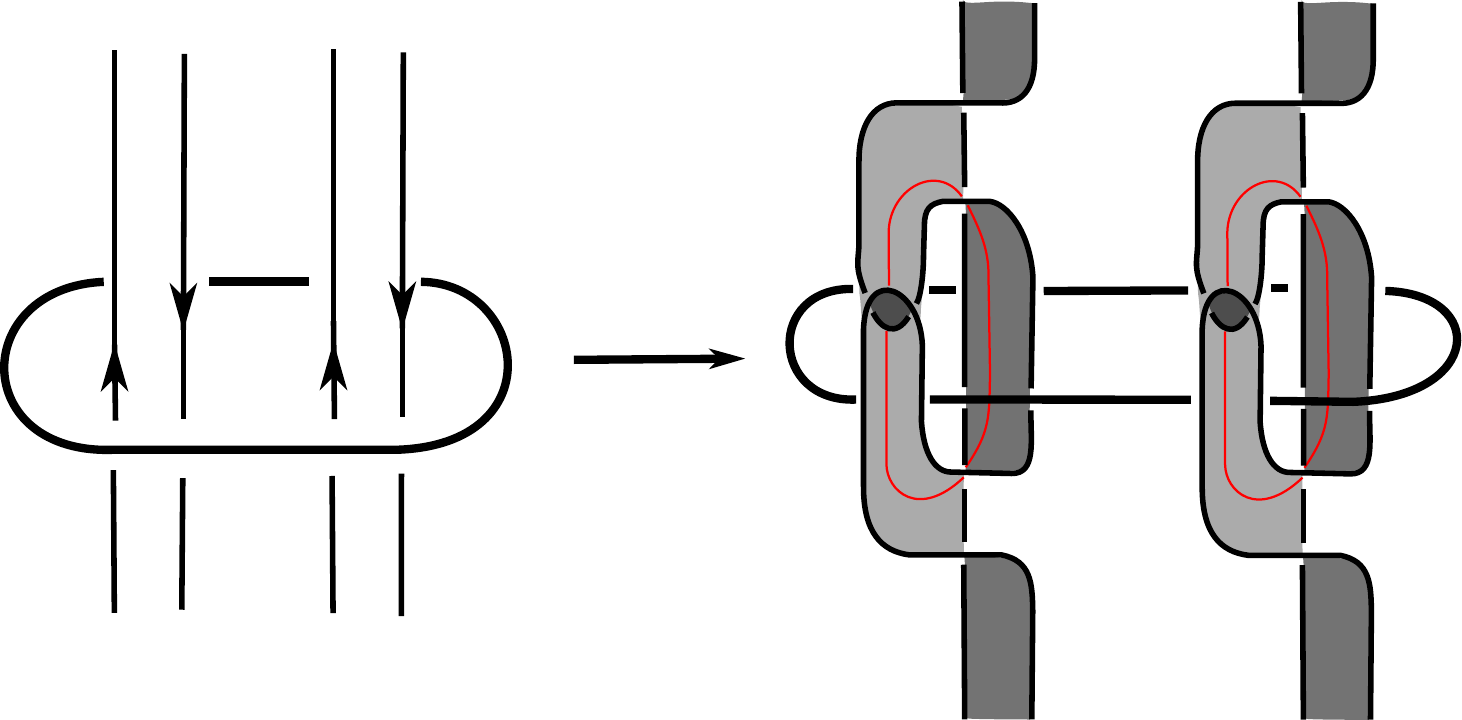}
    \put (33,30) {$C$}
    \put (101,29) {$C$}
    \put (15,23) {{\LARGE $\dots$}}
    \put (74,25) {{\LARGE $\dots$}}
  \end{overpic}
  \caption{Choosing a nice surface to twist. The red curves represent the only homology classes in the basis passing linking with the surgery curve. }
  \label{fig:surface_choice}
  \end{figure}
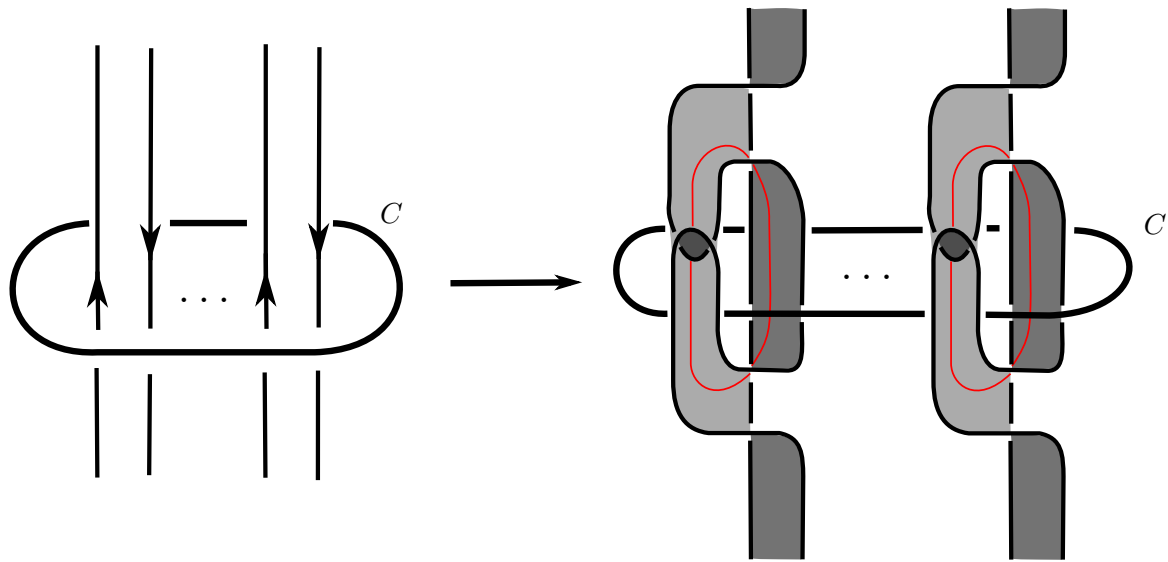  

Thus with respect to an appropriate ordering of the bases, we can assume that $L$ and $L'$ have Seifert matrices $M$ and $M'$ of the form 
\[
M=\begin{pmatrix}\begin{array}{c c c}
\BlockMatrix{0} & B & F_1 \\
C & \BlockMatrix{0} & F_2 \\
F_3 & F_4 & F_5
\end{array}
\end{pmatrix},
\]
and
\[
M'=\begin{pmatrix}\begin{array}{c c c}
\BlockMatrix{-m} & B & F_1 \\
C &
\BlockMatrix{-n} & F_2 \\
F_3 & F_4 & F_5
\end{array}
\end{pmatrix}.
\]
If $-mn$ is a square, then we can assume that $m$ and $n$ take the form $m=-ax^2$ and $n=ay^2$, for some integers $x,y$ and $a$. By stabilizing $M'$ we obtain a new Seifert matrix $M''$ for $L'$:
\[
M''=\begin{pmatrix}\begin{array}{c c c | c c}
\BlockMatrix{ax^2} & B & F_1 
&\ColumnMatrix{-ax}
&\ColumnMatrix{0}
\\
C &
\BlockMatrix{-ay^2} & F_2
&\ColumnMatrix{0}
&\ColumnMatrix{0}\\
F_3 & F_4 & F_5 &\ColumnMatrix{0}
&\ColumnMatrix{0}\\ \hline
\RowMatrix{0} & \RowMatrix{0} & \RowMatrix{0} & 0 & 1\\
\RowMatrix{0} & \RowMatrix{0} & \RowMatrix{0} & 0 & 0

\end{array}
\end{pmatrix}.
\]
Consider the following matrix identity:

{\small
\begin{equation}\label{eq:matrix_id}
\begin{pmatrix}
1& 0& x& 0\\
0& 1& y& ay\\
0& 0& 1& 0\\
0& 0& 0& 1
\end{pmatrix}
\begin{pmatrix}
A+ax^2& B& -ax& 0\\
C& D-ay^2& 0& 0\\
0& 0& 0& 1\\
0& 0& 0& 0
\end{pmatrix}
\begin{pmatrix}
1& 0& 0& 0\\
0& 1& 0& 0\\
x& y& 1& 0\\
0& ay& 0& 1
\end{pmatrix}
=
\begin{pmatrix}
A& B& -ax& x\\
C& D& 0& y\\
0& ay& 0& 1\\
0& 0& 0& 0
\end{pmatrix}.
\end{equation}}
By replacing the entries of the matrices in \eqref{eq:matrix_id} by identity matrices and block matrices of the appropriate size, we see that there is an invertible matrix $P$ such that
\[
P^TM''P=\begin{pmatrix}\begin{array}{c c c | c c}
\BlockMatrix{0} & B & F_1 
&\ColumnMatrix{-ax}
&\ColumnMatrix{x}
\\
C &
\BlockMatrix{0} & F_2
&\ColumnMatrix{0}
&\ColumnMatrix{y}\\
F_3 & F_4 & F_5 &\ColumnMatrix{0}
&\ColumnMatrix{0}\\ \hline
\RowMatrix{0} & \RowMatrix{ay} & \RowMatrix{0} & 0 & 1\\
\RowMatrix{0} & \RowMatrix{0} & \RowMatrix{0} & 0 & 0

\end{array}
\end{pmatrix}.
\]
Since the upper left submatrix of $P^TM''P$ is precisely $M$, this shows that $L'$ has a Seifert matrix obtained by adjoining two additional rows and columns to $M$. Since we started with a surface $F$ realizing the algebraic genus, it follows that
\begin{equation}\label{eq:twist_ub}
\galg(L') \leq \galg(L) +1.
\end{equation}

Since $L'$ can be obtained from $L$ by a null-homologous $-m$-twist and a null-homologous $-n$-twist, we can reverse the roles of $L$ and $L'$
Since the roles of $L$ and $L'$ can be reversed in \eqref{eq:twist_ub}. This shows that
\[
|\galg(L)-\galg(L')|  \leq 1,
\]
as required.
\end{proof}

\section{Decreasing the algebraic genus}\label{sec:decreasing}
Theorem~\ref{thm:alg_twisting} accounts for half of Theorem~\ref{thm:twist_characterization}. In order to complete the proof we need to show that there are always pairs of null-homologous twisting operations that decrease the algebraic genus. This can be done by adapting the argument used by Livingston to prove that any knot can be converted to the unknot using at most $2g$ null-homologous twists \cite{Livingston19twist}.
\begin{prop}\label{prop:alg_twist_decrease}
Given a link $L$ with $\galg(L)>0$, then $L$ can be obtained from a link $L'$ with $\galg(L')=\galg(L)-1$ by a null-homologous $+1$-twist and a null-homologous $-1$-twist.
\end{prop}

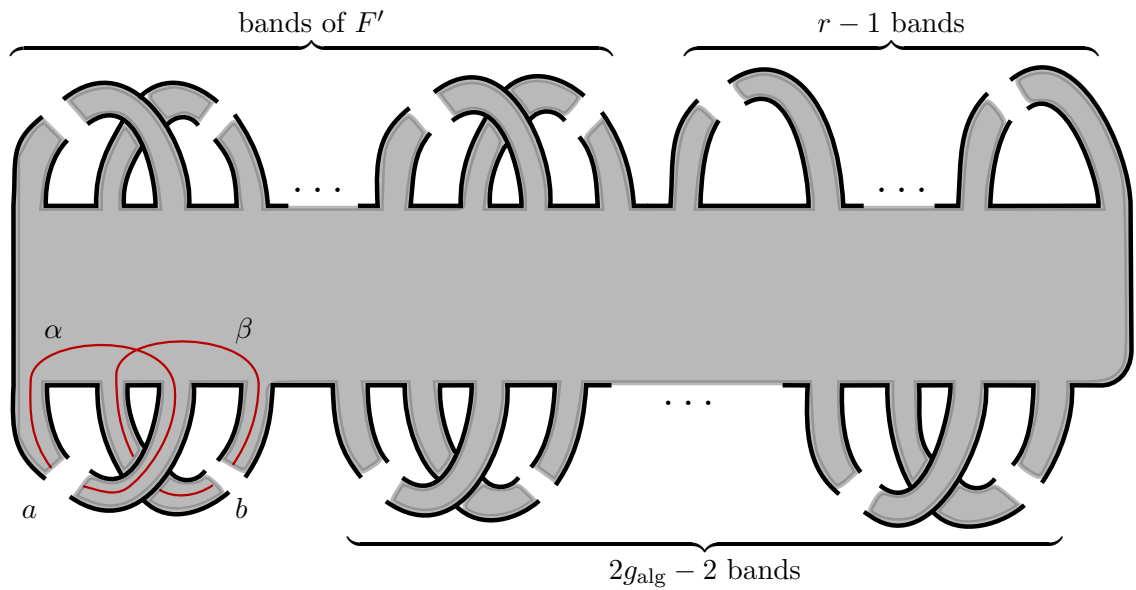
\begin{figure}
\vspace{20pt}
  \begin{overpic}[width=0.9\textwidth]{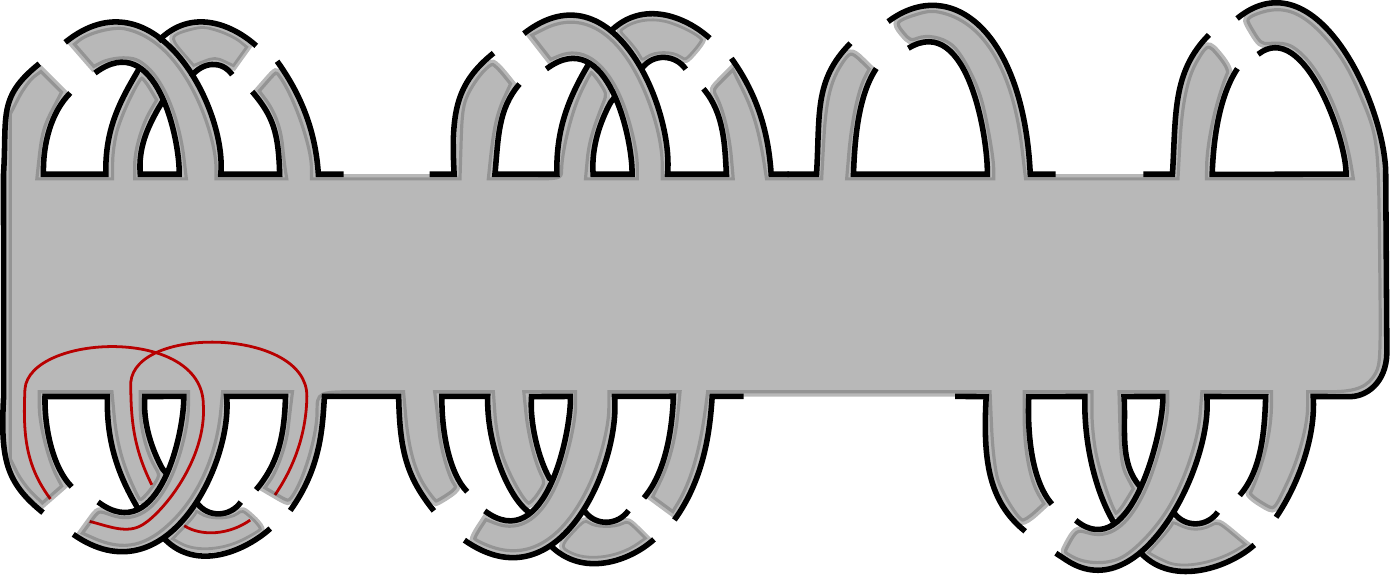}
    \put (0,42) {$\overbrace{\hspace{8cm} }^{\text{{\normalsize bands of $F'$}}}$}
    \put (25,30) {{\LARGE $\dots$}}
    \put (60,42) {$\overbrace{\hspace{5.5cm} }^{\text{{\normalsize $r-1$ bands}}}$}
    \put (77,30) {{\LARGE $\dots$}}
    \put (30,-0.5) {$\underbrace{\hspace{9.5cm} }_{\text{{\normalsize $2\galg-2$ bands}}}$}
    \put (58,11) {{\LARGE $\dots$}}
    \put (3,17) {$\alpha$}
    \put (20,17) {$\beta$}
    \put (1,1) {$a$}
    \put (20,1) {$b$}
  \end{overpic}
  \vspace{20pt}
  \caption{Arranging the handles of the surface $F$. The gaps in the bands indicate that they may be knotted, linked together and twisted.}
  \label{fig:unknotting_surface}
  \end{figure}

\begin{proof}
Suppose that $L$ has $r$ components. Consider a Seifert surface for $F$ for $L$ which realizes the algebraic genus. This contains a connected subsurface $F'$ with $\partial F'$ a knot with Alexander polynomial one and $\galg(L)=g(F)-g(F')$. We may view $F$ as obtained by attaching $2\galg + r-1$ handles to $F'$. So if we present $F'$ as a surface obtained by attaching $2g(F)$ bands to a disk, then we can present $F$ as being obtained by attaching $2\galg(L)+r-1$ further handles to this disk. Furthermore by performing handle slides, we can assume that the bands are grouped together into three groups: the $g(F')$ pairs of bands comprising $F'$, the $r-1$ bands increasing the number of components and the $\galg$ pairs of bands contributing to the algebraic genus of $L$. This is illustrating in Figure~\ref{fig:unknotting_surface}. Let $a$ and $b$ be a pair of handles contributing non-trivially to $\galg(L)$. Let $\alpha$ and $\beta$ be  curves running over the cores of these handles as shown in Figure~\ref{fig:unknotting_surface}. Let $F''$ be the surface obtained by deleting the handles $a$ and $b$ from $F$ and take $L'$ to be the boundary of $F''$. The existence of the surface $F''$ shows that $\galg(L')\leq \galg(L)-1$. However $L$ is obtained from $L'$ by a pair of oriented band moves, so Lemma~\ref{lem:band_addition} shows that $\galg(L')=\galg(L)-1$.  

The aim is to find a pair of null-homologous twists which will transform $L'$ into $L$. We will produce these twists by taking a surgery presentation for $L$ and manipulating it until we find a surgery presentation for $L$ which is a diagram for $L'$ with the addition two appropriately framed surgery curves.

By definition the framing of the curve $\alpha$ is given by $\theta(\alpha,\alpha)$, where $\theta$ is the Seifert form of $F$. We may assume that $\alpha$ has odd framing. If $\beta$ has odd framing, then we  can simply use $b$ in place of $a$. If both $\alpha$ and $\beta$ have even framing, then we can change our handle decomposition of $F$ by sliding $a$ over $b$. After such a slide the curve $\alpha'$ running over $a$ has the homology class of $\alpha+\beta$. This curve has odd framing, since
\begin{align*}
\theta(\alpha+\beta,\alpha+\beta)&= \theta(\alpha,\alpha)+\theta(\beta,\alpha)+\theta(\alpha,\beta)+\theta(\beta,\beta)\\
&\equiv \theta(\beta,\alpha)-\theta(\alpha,\beta) \equiv 1 \bmod 2,
\end{align*}
where we have used that the anti-symmetrization of the Seifert form is the intersection form of $H_1(F;\Z)$ in the second line.

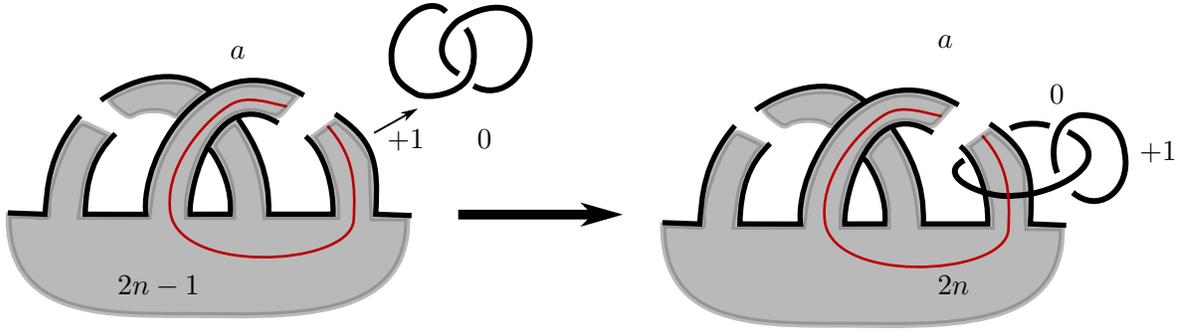
\begin{figure}
  \begin{overpic}[width=0.9\textwidth]{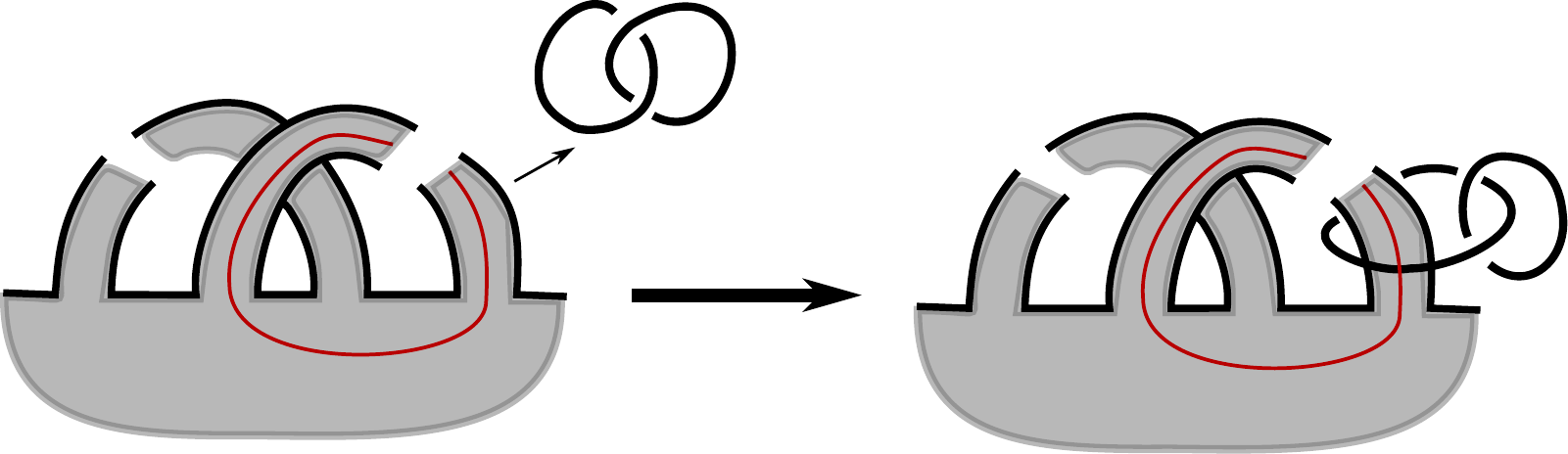}
    \put (10,3) {$2n-1$}
    \put (83,3) {$2n$}
    \put (20,24) {$a$}
    \put (83,25) {$a$}
    \put (42,16) {$0$}
    \put (34,16) {$+1$}
    \put (93,20) {$0$}
    \put (101,15) {$+1$}
  \end{overpic}
  \caption{Sliding the band $a$ over the $+1$-framed component.}
  \label{fig:hopf_over_a}
  \end{figure}

Now we produce our surgery diagram. Introduce a Hopf link to $S^3$ with one component $0$-framed and the other $+1$-framed. This provides a surgery presentation for $S^3$. Slide the band $a$ over the $+1$-framed curve. After this slide, the $0$-framed curve forms a meridian of $a$ and the framing of the core curve $\alpha$ becomes an even integer. Since the $0$-framed curve forms a meridian of $a$, we can slide other bands over it to effect ``crossing changes'' between $a$ and other bands in the handle decomposition of $F$ and also to pass the band $a$ through itself. Thus after some sequence of such moves we can assume that the curve $\alpha$ is unknotted and the band $a$ lies entirely above the band $b$ and is unlinked from all other bands. Moreover notice that sliding $a$ over the $0$-framed curve changes the framing of $\alpha$ by $\pm 2$ and that sliding any other band over $a$ does not change the framing of $\alpha$. Thus the framing on $\alpha$ is still an even integer and moreover by performing further slides we can assume that $\alpha$ has framing $0$. So by performing a sequence of isotopies and handle slides, we can obtain a surface $F'$ where $a$ appears as in Figure~\ref{fig:unknotted_band} but $F'$ is otherwise identical to $F$. Notice that the link $\partial F'$ is isotopic to $L'$, the link bounding the surface $F''$.

\begin{figure}
  \begin{overpic}[width=0.3\textwidth]{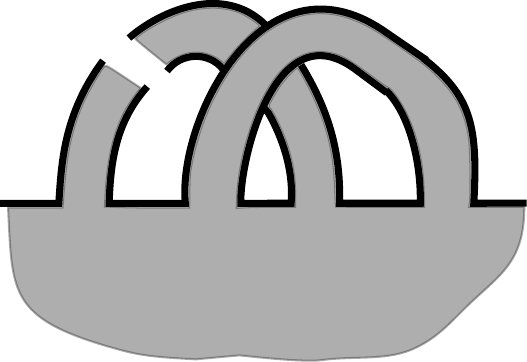}
    \put (95,40) {$a$}
    \put (5,40) {$b$}
  \end{overpic}
  \caption{The band $a$ after simplification.}
  \label{fig:unknotted_band}
  \end{figure}

Now slide the $0$-framed component of the Hopf link over the $+1$-framed component so that it becomes a two component unlink with a $+1$-framed component and a $-1$-framed component. Thus we have a surgery description showing that $L$ can be obtained from $L'$ by perfoming a null-homologous $+1$-twist and a null-homologous $-1$-twist as required.
\begin{figure}
  \begin{overpic}[width=0.9\textwidth]{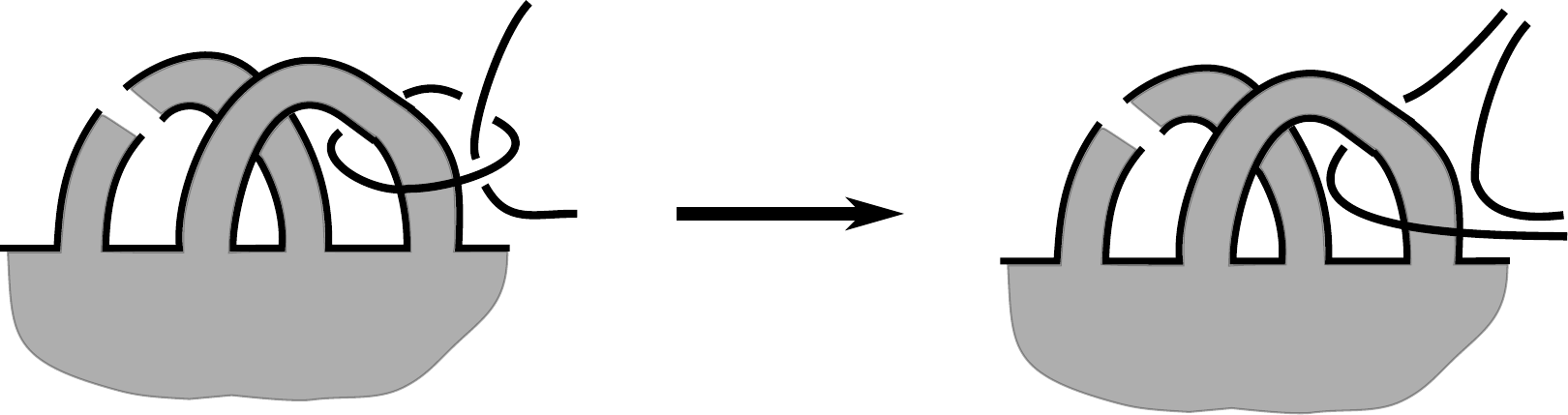}
    \put (20,24) {$a$}
    \put (83,23) {$a$}
    \put (27,22) {$0$}
    \put (35,23) {$+1$}
    \put (90,24) {$-1$}
    \put (97,15) {$+1$}
  \end{overpic}
  \caption{Sliding the $0$-framed component over the $+1$-framed component.}
  \label{fig:sliding_to_unlink}
  \end{figure}
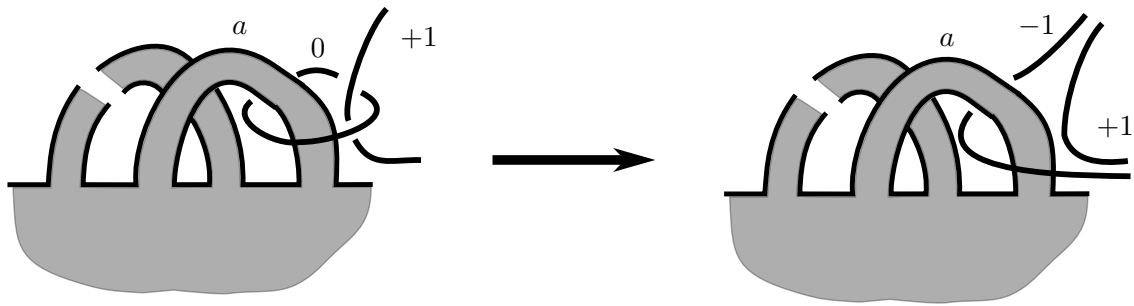

\end{proof}
Thus we can prove Theorem~\ref{thm:twist_characterization}.
\twistcharacterization*

\begin{proof}
Theorem~\ref{thm:alg_twisting} shows that if $L$ is obtained from $L'$ with $\galg(L')=0$ by $p$ null-homologous $+1$-twists and $n$ null-homologous twists $-1$-twists, then $\galg(L)\leq \max\{n,p\}$. On the other hand, applying Proposition~\ref{prop:alg_twist_decrease} repeatedly shows that $L$ can be converted into a link $L'$ with $\galg(L')=0$, by $\galg(L)$ pairs of null-homologous $+1$-twists and $-1$-twists.
\end{proof}

\section{Satellite knots}\label{sec:satellite}
In this section we prove Theorem~\ref{thm:satellites}. First we note that null-homologous twisting is preserved under satellite operations.
\begin{lem}\label{lem:satellite_move}
Let $K$ and $K'$ be knots related by a null-homologous $n$-twist, then for any pattern $P\subseteq S^1\times D^2$, the satellite knots $P(K)$ and $P(K')$ are related by a null-homologous $n$-twist.
\end{lem}
\begin{proof}
Let $X_P$ denote the complement $X_P=S^1\times D^2\setminus \nu P$ which comes with a meridian $\mu$ and distinguished longitude $\lambda$ in $\partial( S^1\times D^2)$. The knot complement $S^3\setminus \nu P(K)$ is obtained by gluing $X_P$ to $S^3\setminus \nu K$ so that $\mu$ and $\lambda$ are glued to the meridian and null-homologous longitude of $K$ respectively. The complement $S^3\setminus \nu P(K')$ is constructed similarly by gluing $X_P$ to $S^3\setminus \nu K'$.

Since $K$ and $K'$ are related by a null-homologous $n$-twist there is a null-homologous curve $C \subset S^3\setminus \nu K$ which such that performing $1/n$ surgery yields $S^3\setminus \nu K'$. Since $C$ is null-homologous in $S^3\setminus \nu K$, surgering $C$ takes the meridian and null-homologous longitude of $S^3\setminus \nu K$ to the meridian and null-homologous longitude of $S^3\setminus \nu K'$. Thus if we consider $C$ as a curve in $S^3\setminus \nu P(K)= S^3\setminus \nu K \cup X_P$, we see that $1/n$ surgery on $C$ will produce $S^3\setminus \nu P(K')$. Since $C$ is null-homologous in $S^3\setminus \nu K$ it is null-homologous in $S^3\setminus \nu P(K)$, thus $P(K)$ and $P(K')$ are related by a null-homologous $n$-twist.
\end{proof}

\begin{lem}\label{lem:alg_triv_satellite}
Let $K'$ be a knot with $\Delta_{K'}(t)=1$, then for any pattern $P$, we have $\galg(P(K'))= \galg(P(U))$.
\end{lem}
\begin{proof}
We refer the reader to \cite[Proof of Theorem~6.15]{Lickorish1997introduction}. In this proof, Lickorish constructs a Seifert matrix for $P(K')$ of the form $\begin{pmatrix}
M &0\\
0 & X
\end{pmatrix}$
where $M$ is a Seifert matrix for $P(U)$ and $X$ is a matrix satisfying
$\det (tX-X^T) = \Delta_{K'}(t^w)$, where $w$ is the winding number of $P$. Since $\Delta_{K'}(t)=1$, this shows that $P(K')$ and $P(U)$ are $S$-equivalent, and hence have the same algebraic genus.
\end{proof}

\satellite*
\begin{proof}
By Proposition~\ref{prop:alg_twist_decrease}, $K$ can be converted into a knot $K'$ with Alexander polynomial one by a sequence of at most $\galg(K)$ pairs of null-homologus $+1$-twists and $-1$-twists. By Lemma~\ref{lem:satellite_move} this shows that $P(K)$ can be converted to $P(K')$ by a similar sequence of twists. Thus we have
\[
\galg(P(K))\leq \galg(K) + \galg(P(K')).
\]
By Lemma~\ref{lem:alg_triv_satellite} we have $\galg(P(K'))=\galg(P(U))$ so this is the desired bound.
\end{proof}

\section{Torus knots}\label{sec:torus}
We now gather the ingredients to prove Theorem~\ref{thm:top_genus}. 
\begin{lem}\label{lem:2power_untwisting}
For any $a,b\geq 1$, we have
\[\galg(T_{2^a, 2^b})< \frac{2^{a+b}}{3}.\]
\end{lem}
\begin{proof}
By Theorem~\ref{thm:alg_twisting} it suffices to prove $T_{2^a, 2^b}$ can be converted to the unlink using at most $\frac{2^{a+b}}{3}$ null-homologous twists. Given a full twist on $2^{k+1}$ strands oriented so that all crossings are positive, we can perform a null-homologous $-1$-twist to produce two parallel sets of $2^k$ strands each with two positive full twists. This is depicted in Figure~\ref{fig:undo_full_twist}. 
Thus if we let $T_{k}$ denote the number null-homologous twisting moves required to undo a full twist on $2^k$ strands we see that $T_k$ satisfies the recursive bound $T_{k+1}\leq 1+ 4T_{k}$. Note that $T_1=1$ since a full twist on two strands can be undone by a single crossing change. Now the solution to the recursion relation $c_{k+1}=1+4c_{k}$ with $c_1=1$ is $c_k=\frac{4^k-1}{3}$. Thus we see that $T_{k}\leq \frac{4^k-1}{3}$ for all $k$.

Without loss of generality suppose that $2^a\leq 2^b$. The link $T_{2^a,2^b}$ can be viewed as $2^{b-a}$ full twists on $2^a$ strands.  Thus $T_{2^a, 2^b}$ can be converted into the unlink $T_{2^a,0}$ by removing $2^{b-a}$ positive full twists on $2^a$ strands. Thus
\[\galg(T_{2^a, 2^b})\leq 2^{b-a}\times\frac{4^a-1}{3}< \frac{2^{a+b}}{3},\]
as required.

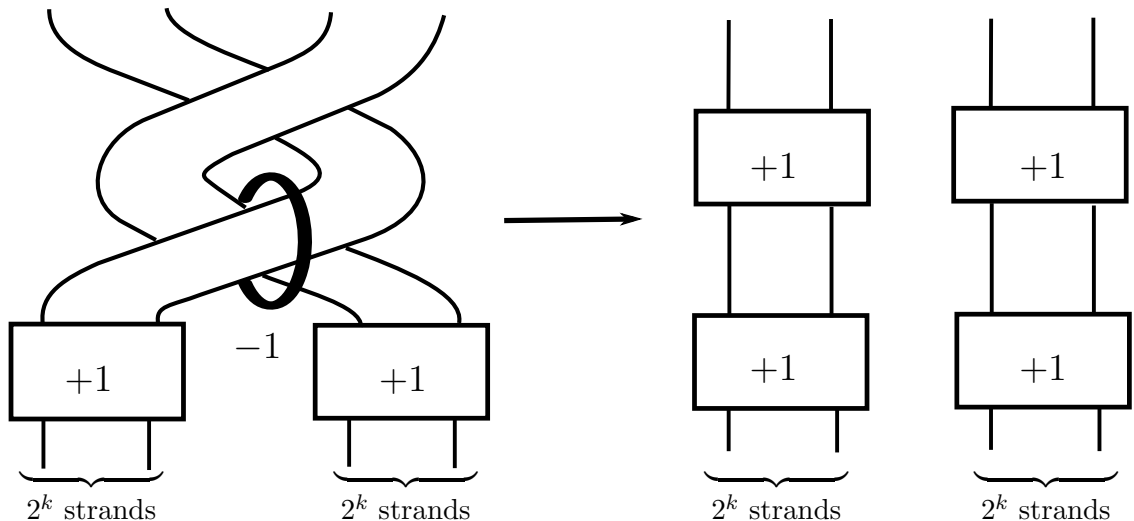
\begin{figure}
  \begin{overpic}[width=0.9\textwidth]{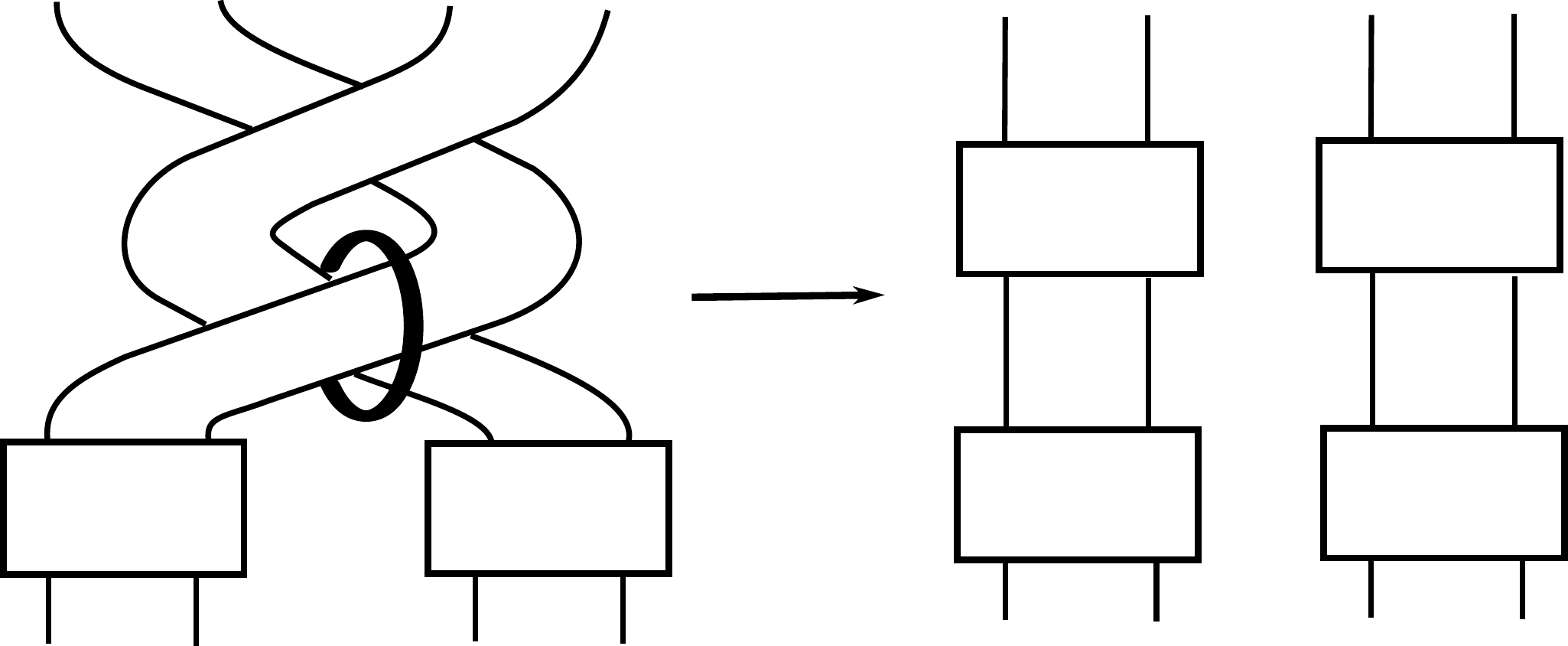}
    \put (20,10) {{\Large $-1$}}
    \put (1,0) {$\underbrace{\hspace{1.9cm} }_{\text{\normalsize $2^k$ strands}}$}
    \put (29,0) {$\underbrace{\hspace{1.9cm} }_{\text{\normalsize $2^k$ strands}}$}
    \put (62,0) {$\underbrace{\hspace{1.9cm} }_{\text{\normalsize $2^k$ strands}}$}
    \put (86,0) {$\underbrace{\hspace{1.9cm} }_{\text{\normalsize $2^k$ strands}}$}
    \put (5,7) {{\Large $+1$}}
    \put (33,7) {{\Large $+1$}}
    \put (66,8) {{\Large $+1$}}
    \put (90,8) {{\Large $+1$}}
    \put (66,26) {{\Large $+1$}}
    \put (90,26) {{\Large $+1$}}
  \end{overpic}
  \vspace{20pt}
  \caption{Converting a full twist on $2^{k+1}$ strands into four full twists on $2^k$ strands with a single null-homologous twist. Each box contains a full twist.}
  \label{fig:undo_full_twist}
  \end{figure}

\end{proof}

\begin{lem}\label{lem:surface_gluing}
For any $a,b,c\geq 1$,
\[\galg(T_{a,b+c})\leq \galg(T_{a,b})+\galg(T_{a,c})+a.
\]
\end{lem}
\begin{proof}
Observe that $T_{a,b+c}$ can be converted into the split link $T_{a,b}\sqcup T_{a,c}$ by performing $a$ oriented crossing resolutions: if one considers $T_{a,b+c}$ as the closure of the braid word $(\sigma_1 \dotsb \sigma_{b+c-1})^a$, then these resolutions corresponding to deleting all instances of $\sigma_b$ from this braid word. Thus $T_{a,b+c}$ can be obtained from the split link $T_{a,b}\sqcup T_{a,c}$ by $a$ oriented band moves. Thus, Lemma~\ref{lem:alg_band_attach} implies that
\[
\galg(T_{a,b+c})\leq \galg(T_{a,b})+\galg(T_{a,c})+a,
\]
as required.
\end{proof}

We piece together the bounds from Lemma~\ref{lem:2power_untwisting} and Lemma~\ref{lem:surface_gluing} to obtain Theorem~\ref{thm:top_genus}.
\topgenus*
\begin{proof}
Suppose that $p$ and $q$ are represented in binary as $\sum_{i=0}^k 2^{a_i} =q$ and $p=\sum_{j=0}^l 2^{b_j}$, i.e. so that when represented in binary $q$ has $k+1$ non-zero digits and $p$ has $l+1$ non-zero digits when represented in binary. Notice that we have
\begin{equation}\label{eq:kl_bounds}
k\leq\log_2 p \quad\text{and}\quad l\leq \log_2 q.
\end{equation}
For any given any $i$ and $j$, Lemma~\ref{lem:2power_untwisting} shows that the algebraic genus of the link $L_{i,j}:=T_{2^{b_j}, 2^{a_i}}$ satisfies
\[\galg(L_{i,j})< \frac{2^{a_i+b_j}}{3}.
\]
By applying Lemma~\ref{lem:surface_gluing} to $L_{0,j},\dots, L_{k,j}$, we see that $T_{q,2^{b_j}}$ satisfies
\begin{align*}
\galg(T_{q,2^{b_j}})&< 2^{b_j}k + \sum_{i=0}^k \left(\frac{2^{a_i+b_j}}{3}\right)\\
&=2^{b^j}k +\frac{2^{b_j}q}{3}.
\end{align*}

So by applying Lemma~\ref{lem:surface_gluing} to $T_{q,2^{b_0}},\dots, T_{q,2^{b_l}}$, we see that $T_{p,q}$ satisfies
\begin{align*}
\galg(T_{p,q})&< lq + \sum_{j=0}^l \left(\frac{2^{b_j}q}{3}+2^{b_j}k\right)\\
&=\frac{pq}{3}+ ql + pk .
\end{align*}
By \eqref{eq:kl_bounds}, this shows that
\[
\galg(T_{p,q}) < \frac{pq}{3}+p\log_2 q +q\log_2 p,
\]
as required.
\end{proof}

\section{Anisotropic Seifert forms}\label{sec:anisotropic}
In this section we prove Proposition~\ref{prop:other_twisting} which shows that shows that most pairs of null-homologous twisting operations can change the algebraic genus and the topological slice genus by two. Recall that a quadratic form $q: \Z^n \rightarrow \Z$ is {\em isotropic} if there is $v\neq 0$ such that $q(v)=0$ and {\em anisotropic} otherwise.
\begin{lem}\label{lem:taylor_application}
Let $K$ be a knot with a Seifert surface $F$ and associated Seifert form $\theta$. If $\gtop(K)<g(F)$, then the quadratic form on $H_1(F;\Z)$ defined by $v\mapsto \theta(v,v)$ is isotropic.
\end{lem}
\begin{proof}
Given a knot with a genus $g$ Seifert surface $F$ and Seifert form $\theta: \Z^{2g} \times \Z^{2g}\rightarrow \Z$, Taylor defines a knot invariant \cite{Taylor79genera}:
\[
t(K):= g-a(\theta),
\]
where $a(\theta)$ is the rank of a maximal isotropic subgroup of $\Z^{2g}$ (i.e. the maximal rank of a subgroup on which $\theta$ is identically 0). As discussed in \cite[Section~2]{Lewark2019calculating}, this invariant is known to be a lower bound for the topological slice genus. In particular, we have $a(\theta)\geq g(F)-\gtop(K)$. Thus if $\gtop(K)<g(F)$, then $\theta$ has a non-trivial isotropic subgroup, as required.
\end{proof}
In order to apply Lemma~\ref{lem:taylor_application} we will need to show certain forms are anisotropic. This requires some elementary number theory. For a prime $p$, we use $\Legendre{n}{p}$ to denote the Legendre symbol of $n$ modulo $p$. 

\begin{lem}\label{lem:anisotropic_condition}
Let $p$ be an odd prime and let $a,b,M,N$ be positive integers coprime to $p$. The quadratic form
\[
q(x_1, \dots , x_4)= ax_1^2- bx_2^2+ p(Mx_3^2+Nx_4^2)
\]
is anisotropic if $\Legendre{ab}{p}=-1$.
\end{lem}
\begin{proof}
We will show that if $q$ is isotropic, then $\Legendre{ab}{p}=1$. If $q$ is isotropic, then there are integers $y_1,\dots, y_4$ such that $\gcd(y_1,\dots, y_4)=1$ and
\begin{equation}\label{eq:iso_solution}
ay_2^2-by_1^2= p(My_3^2+Ny_4^2).
\end{equation}
Since $p$ divides the right hand side, we see that $y_1$ and $y_2$ provide a solution to the equation $aX^2\equiv bY^2\bmod p$. Moreover, this is a non-trivial solution, that is $y_1,y_2\not\equiv 0 \bmod p$. Assume for sake of contradiction that $y_1\equiv y_2\equiv 0\bmod p$. If both sides of \eqref{eq:iso_solution} are non-zero, then the largest power of $p$ dividing the left hand side is even, but the largest power of $p$ dividing the right hand side is odd. Thus both sides of \eqref{eq:iso_solution} must be zero. This implies that $y_3=y_4=0$, which would imply that $\gcd(y_1,\dots, y_4)\geq p>1$. Thus we must have $y_1,y_2\not\equiv 0 \bmod p$.

Since the quadratic residues form an index two subgroup in $(\Z/p\Z)^\times$, the equation $aX^2\equiv bY^2 \bmod p$ has a non-trivial solution if and only if both $a$ and $b$ are quadratic residues or  both $a$ and $b$ are quadratic non-residues modulo $p$. In either case, this implies that if $aX^2\equiv bY^2 \bmod p$ has a non-trivial solution, then $\Legendre{ab}{p}=1$, as required.
\end{proof}

\begin{lem}\label{lem:non_trivial_legendre}
For any integer $n>0$ which is not a square, there is an odd prime $p$ such that
\[
\Legendre{n}{p}=-1.
\]
\end{lem}
\begin{proof}
This is a standard application of quadratic reciprocity and Dirichlet's theorem on primes in arithmetic progressions.
Suppose that $n$ has prime factorization $n=p_1^{a_1} \dots p_k^{a_k}$. Since $n$ is not a square, at least one of the $a_i$ is odd. Without loss of generality assume that $a_1$ is odd. Suppose first that $p_1$ is an odd prime. By the Chinese remainder theorem and Dirichlet's theorem on primes in arithmetic progressions, we can choose a prime $p$ satisfying the congruences $p\equiv 1 \bmod 4$, $p\equiv 1\bmod p_i$ for $i>1$ and $p\equiv q\bmod p_1$, where $q$ satisfies $\Legendre{q}{p_1}=-1$. It follows from quadratic reciprocity that such a $p$ satisfies $\Legendre{n}{p}=-1$.

If $p_1=2$, then we choose $p$ to be a prime satisfying $p\equiv 5 \bmod 8$ and $p\equiv 1\bmod p_i$ for all $i>1$. Using quadratic reciprocity and the fact that $\Legendre{2}{p}=-1$ for $p\equiv 5 \bmod 8$, we see that such a $p$ satisfies $\Legendre{n}{p}=-1$.
\end{proof}

\othertwisting*
\begin{proof}
\begin{figure}
\vspace{10pt}
  \begin{overpic}[width=0.6\textwidth]{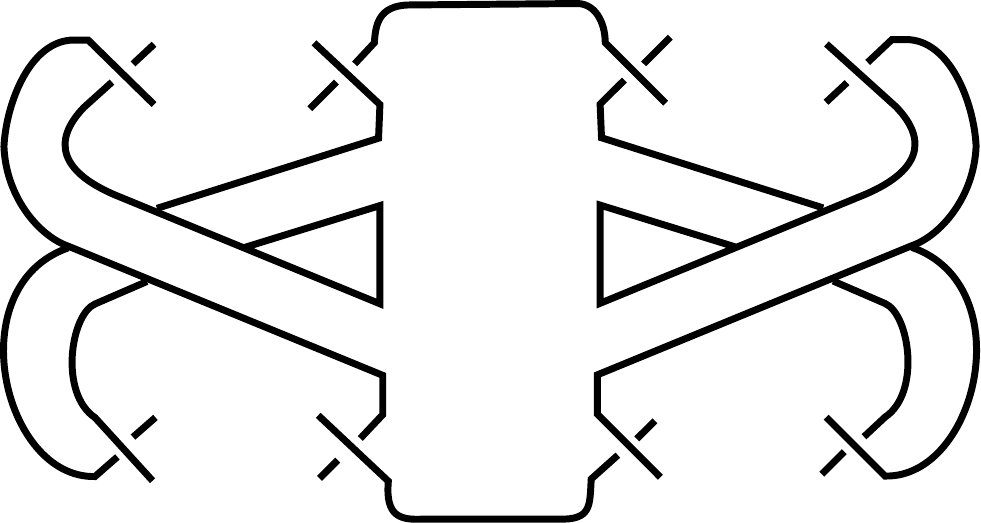}
    \put (10,50) {$\overbrace{\hspace{2.5cm} }^{\text{{\normalsize $a$ twists}}}$}
    \put (63,50) {$\overbrace{\hspace{2.5cm} }^{\text{{\normalsize $b$ twists}}}$}
    \put (10,2) {$\underbrace{\hspace{2.5cm} }_{\text{{\normalsize $c$ twists}}}$}
    \put (63,2) {$\underbrace{\hspace{2.5cm} }_{\text{{\normalsize $d$ twists}}}$}
    \put (20,45) {{\LARGE $\dots$}}
    \put (72,7) {{\LARGE $\dots$}}
    \put (20,7) {{\LARGE $\dots$}}
    \put (72,45) {{\LARGE $\dots$}}
  \end{overpic}
  \vspace{10pt}
  \caption{The knot $K(a,b,c,d)$.}
  \label{fig:k_abcd}
  \end{figure}
  
  \begin{figure}
  \begin{overpic}[width=0.6\textwidth]{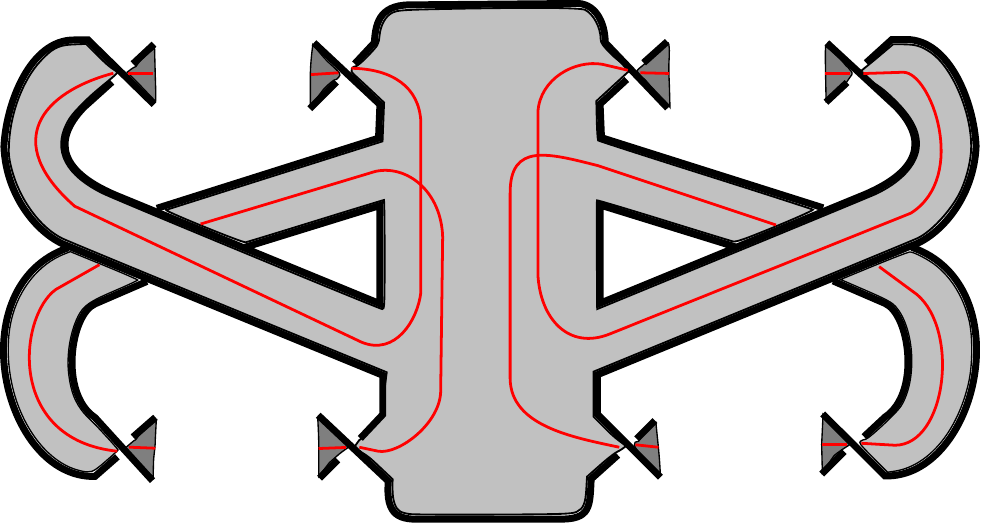}
    \put (40,47) {$\alpha_1$}
    \put (55,48) {$\alpha_2$}
    \put (42,5) {$\alpha_3$}
    \put (54,6) {$\alpha_4$}
    \put (20,45) {{\LARGE $\dots$}}
    \put (72,7) {{\LARGE $\dots$}}
    \put (20,7) {{\LARGE $\dots$}}
    \put (72,45) {{\LARGE $\dots$}}
  \end{overpic}
  \caption{A Seifert surface for $K(a,b,c,d)$.}
  \label{fig:k_abcd_surface}
  \end{figure}
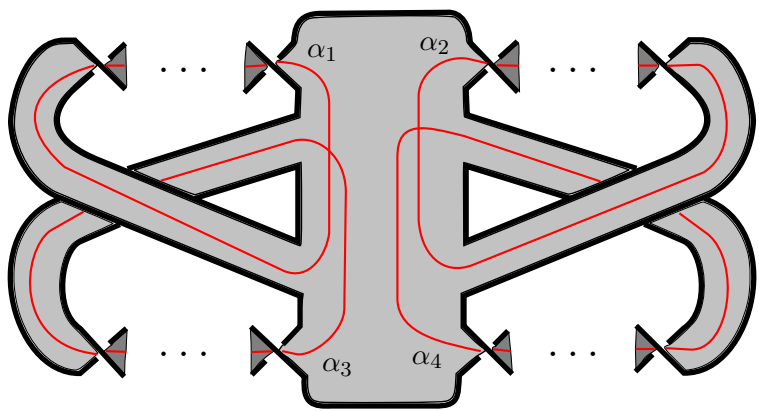
For integers $a,b,c,d$ let $K=K(a,b,c,d)$ be the knot as shown in Figure~\ref{fig:k_abcd}. With respect to the Seifert surface and basis shown in Figure~\ref{fig:k_abcd_surface}, this has Seifert matrix 
\begin{equation}\label{eq:Kabcd_matrix}
M=\begin{pmatrix}
a & 0 & 1& 0\\
0 & b & 0& 1\\
0 & 0 & c& 0\\
0 & 0 & 0& d\\
\end{pmatrix}.
\end{equation}

Notice that for any $c$ and $d$ the knot $K(0,0,c,d)$ is the unknot. Thus we see that $K=K(a,b,c,d)$ can be unknotted by performing a null-homologous $a$-twist and a null-homologous $b$-twist. The aim is to show that for any $a,b$ such that $-ab$ is not a square, then we can find $c$ and $d$ such that $\gtop(K(a,b,c,d))=2$.
Without loss of generality, assume that $a>0$. First suppose that we also have $b>0$. In this case, one can easily see from \eqref{eq:Kabcd_matrix} that $\sigma(K(a,b,c,d))=4$ for any $c>0$ and $d>0$.

Thus it suffices to consider the case where $a>0$ and $b<0$. By Lemma~\ref{lem:taylor_application} it suffices to find $c,d$ ensuring that the Seifert form is anisotropic. That is we need to show that the quadratic form \eqref{eq:Kabcd_matrix}
\begin{equation}
q(x_1, \dots, x_4)= ax_1^2+ x_1x_3 + cx_3^2 -|b|x_2^2+ x_2x_4 + cx_4^2
\end{equation}
is not always isotropic. This can be diagonalized over $\Q$ as
\begin{equation}
q(x_1, \dots, x_4)=a\left(x_1 + \frac{x_3}{2a}\right)^2+a(4ac-1)\left(\frac{x_3}{2a}\right)^2 -|b|\left(x_2 - \frac{x_4}{2|b|}\right)^2+|b|(4|b|d+1)\left(\frac{x_4}{2|b|}\right)^2.
\end{equation}
Since a quadratic form is isotropic over $\Z$ if and only if it is isotropic over $\Q$. It suffices to show that the form
\[
\widetilde{q}(x_1, \dots, x_4)=ax_1^2-|b|x_2^2+a(4ac-1)x_3^2 +|b|(4|b|d+1)x_4^2
\]
is anisotropic for some choice of $c$ and $d$. Since we are assuming that $-ab$ is not a square, then Lemma~\ref{lem:non_trivial_legendre} shows there is an odd prime $p$ with $\left(\frac{a|b|}{p}\right)=-1$. Since $p$ is coprime to $a$ and $b$, we can find $c$ and $d$ such that $p$ divides both $(4ac-1)$ and $(4|b|d+1)$ but $p^2$ does not divide either $(4ac-1)$ or $(4|b|d+1)$. Lemma~\ref{lem:anisotropic_condition} shows that for such $c$ and $d$ the Seifert form is anisotropic and hence $\gtop(K(a,b,c,d))=2$, as required.
\end{proof}

\bibliography{alg_genus_bib}

\begin{thebibliography}{FMPC19}

\bibitem[BF14]{Borodzik14Algebraic}
Maciej Borodzik and Stefan Friedl.
\newblock On the algebraic unknotting number.
\newblock {\em Trans. London Math. Soc.}, 1(1):57--84, 2014.

\bibitem[BF15]{Borodzik15UnknottingI}
Maciej Borodzik and Stefan Friedl.
\newblock The unknotting number and classical invariants, {I}.
\newblock {\em Algebr. Geom. Topol.}, 15(1):85--135, 2015.

\bibitem[BFLL18]{Baader18topological}
S.~Baader, P.~Feller, L.~Lewark, and L.~Liechti.
\newblock On the topological 4-genus of torus knots.
\newblock {\em Trans. Amer. Math. Soc.}, 370(4):2639--2656, 2018.

\bibitem[BL17]{Baader17stable_alternating}
Sebastian Baader and Lukas Lewark.
\newblock The stable 4-genus of alternating knots.
\newblock {\em Asian J. Math.}, 21(6):1183--1190, 2017.

\bibitem[Fel16]{Feller16degree}
Peter Feller.
\newblock The degree of the {A}lexander polynomial is an upper bound for the
  topological slice genus.
\newblock {\em Geom. Topol.}, 20(3):1763--1771, 2016.

\bibitem[FL18]{Feller18Classical}
Peter Feller and Lukas Lewark.
\newblock On classical upper bounds for slice genera.
\newblock {\em Selecta Math. (N.S.)}, 24(5):4885--4916, 2018.

\bibitem[FL19]{Feller19balanced}
Peter Feller and Lukas Lewark.
\newblock Balanced algebraic unknotting, linking forms, and surfaces in three-
  and four-space.
\newblock {\em arXiv:1905.08305}, 2019.

\bibitem[FM16]{Feller16twobridge}
Peter Feller and Duncan McCoy.
\newblock On 2-bridge knots with differing smooth and topological slice genera.
\newblock {\em Proc. Amer. Math. Soc.}, 144(12):5435--5442, 2016.

\bibitem[FMPC19]{Feller19satellite}
P.~Feller, A.~N. Miller, and J.~Pinzon-Caicedo.
\newblock A note on the topological slice genus of satellite knots.
\newblock {\em In preparation}, 2019.

\bibitem[Fre82]{Freedman82topology}
Michael~Hartley Freedman.
\newblock The topology of four-dimensional manifolds.
\newblock {\em J. Differential Geom.}, 17(3):357--453, 1982.

\bibitem[Lic97]{Lickorish1997introduction}
W.B.~Raymond Lickorish.
\newblock {\em An Introduction to Knot Theory}.
\newblock Springer, 1997.

\bibitem[Lie16]{Liechti16positive}
Livio Liechti.
\newblock Positive braid knots of maximal topological 4-genus.
\newblock {\em Math. Proc. Cambridge Philos. Soc.}, 161(3):559--568, 2016.

\bibitem[Liv19]{Livingston19twist}
Charles Livingston.
\newblock Null-homologous unknottings.
\newblock {\em arXiv:1902.05405}, 2019.

\bibitem[LM19]{Lewark2019calculating}
Lukas Lewark and Duncan McCoy.
\newblock On calculating the slice genera of 11- and 12-crossing knots.
\newblock {\em Exp. Math.}, 28(1):81--94, 2019.

\bibitem[Rud84]{Rudolph84Projective}
Lee Rudolph.
\newblock Some topologically locally-flat surfaces in the complex projective
  plane.
\newblock {\em Comment. Math. Helv.}, 59(4):592--599, 1984.

\bibitem[Tay79]{Taylor79genera}
L.~R. Taylor.
\newblock On the genera of knots.
\newblock In {\em Topology of low-dimensional manifolds ({P}roc. {S}econd
  {S}ussex {C}onf., {C}helwood {G}ate, 1977)}, volume 722 of {\em Lecture Notes
  in Math.}, pages 144--154. Springer, Berlin, 1979.

\end{thebibliography}
\bibliographystyle{alpha}

\end{document}